\documentclass[11pt]{amsart}
\usepackage{amssymb,latexsym,amsmath,amsthm}
\usepackage{todonotes}
\usepackage[colorlinks=true,linkcolor=red,urlcolor=blue,citecolor=blue]{hyperref}
\usepackage[latin1]{inputenc}
\newtheorem{thm}{Theorem}[section]
\newtheorem{prop}{Proposition}[section]
\newtheorem{cor}[thm]{Corollary}
\newtheorem{lem}[thm]{Lemma}

\newtheorem{rem}[thm]{Remark}

\numberwithin{equation}{section}

\renewcommand{\Re}{\mathbb R}
\renewcommand{\epsilon}{\varepsilon}
\newcommand{\Ren}{\Re^n}
\newcommand{\Se}{\mathbb S}

\renewcommand{\epsilon}{\varepsilon}
\renewcommand{\phi}{\varphi}

\newcommand{\norm}[1]{\|#1\|}
\newcommand{\normbl}[1]{\left\|#1\right\|_L}

\newcommand{\normk}[1]{\|#1\|_{K}}
\newcommand{\norml}[1]{\|#1\|_{L}}
\newcommand{\st}{\colon}

\DeclareMathOperator{\dist}{dist}

\begin{document}
\title{PUSH FORWARD MEASURES AND CONCENTRATION PHENOMENA}

 \author[PUSH FORWARD MEASURES AND CONCENTRATION PHENOMENA]{C. Hugo~Jim\'{e}nez${}^\dag$${}^\ddag$\and M\'arton ~Nasz\'odi${}^*$\and Rafael~Villa ${}^\dag$}

 \address{C. Hugo~Jim\'{e}nez\and Rafael~Villa, Universidad de Sevilla, Departamento de An\'{a}lisis Matem\'{a}tico, Apartado de Correos 1160,
  Sevilla, Spain 41080} 
\email{carloshugo@us.es,villa@us.es}
 \address{M\'arton ~Nasz\'odi,E\"{o}tv\"{o}s University,
P\'{a}zm\'{a}ny P\'{e}ter S\'{e}t\'{a}ny 1/C
Budapest, Hungary 1117}
 \email{nmarci@math.elte.hu}

\thanks{${}^\dag$ Authors partially supported by the Spanish Ministry of Science and Innovation grant MTM2009-08934 and by the Junta de Andaluc\'{i}a, grant P08-FQM-03543.}
 \thanks{${}^\ddag$ Author partially supported by CONACyT}
\thanks{${}^*$ Author partially supported by the Hung. Nat. Sci. Found. (OTKA) grant no. K72537.}
\maketitle
\begin{abstract}
In this note we study how a concentration phenomenon can be transmitted from one measure $\mu$ to a push-forward measure $\nu$. In the first part, we push forward $\mu$ by $\pi:supp(\mu)\rightarrow \Ren$, where $\pi x=\frac{x}{\norm{x}_L}\norm{x}_K$, and obtain a concentration inequality in terms of the medians of the given norms (with respect to $\mu$) and the Banach-Mazur distance between them. This approach is finer than simply bounding the concentration of the push forward measure in terms of the Banach-Mazur distance between $K$ and $L$. As a corollary we show that any normed probability space with good concentration is far from any high dimensional subspace of the cube. In the second part, two measures $\mu$ and $\nu$ are given, both related to the norm $\norm{\cdot}_L$, obtaining a concentration inequality in which it is involved the Banach-Mazur distance between $K$ and $L$ and the Lipschitz constant of the map that pushes forward $\mu$ into $\nu$. As an application, we obtain a concentration inequality for the cross polytope with respect to the normalized Lebesgue measure and the $\ell_1$ norm.
\end{abstract}

\section{Introduction}

The concentration of measure phenomenon was first studied by L\'{e}vy in \cite{Lev}, for the Haar measure on Euclidian spheres. The definition of the concentration function was first introduced in \cite{Am-Mil}. It was formalized in \cite{G-M} and further analyzed in \cite{M-S}. A comprehensive survey, cf. \cite{Led}.

In this note, for two $n$-dimensional normed spaces $X$ and $Y$, we pose the problem to transmit a concentration of measure phenomenon from $X$ to $Y$. A starting point in this direction is Proposition \ref{prop:dec}, showing that concentration inequalities can be preserved through Lipschitz maps. This result can be applied to push forward the Gaussian measure to the Euclidean unit ball (see \cite{Led}).

Bobkov and Ledoux studied the transference of concentration of measure and other properties via push forward measures. In \cite{Bo-Le} the authors recover concentration results, in the form of Logarithmic Sobolev inequalities, in a uniformly convex space $X$ (from \cite{G-M2,A-B-V}) by pushing forward an adequate measure into the uniform distribution on the unit ball of $X$.  Caffarelli's contraction Theorem \cite{Caf1} and its many applications have been extensively used to transfer concentration and other properties from the Gaussian measure to measures with densities of the form $exp(-V(x))dx$ with $V:\Ren\rightarrow \Re$ satisfying $V''(x)\geq c I_n$ for $c>0$. Thus, important results from Bakry-\'{E}mery \cite{B-E}, Bakry-Ledoux \cite{Ba-L} and others can be recovered. The described approach along with Brenier's optimal transport map \cite{Bre} have been exploited by many authors (e.g. \cite{Cord,Lat-Woj}) to extend properties that were known on measures such as the Gaussian or the Exponential to more general measures. For stability of concentration and other properties under maps that are Lipschitz only on average we refer to the work of E. Milman \cite{MilE} and for a very recent extension of Caffarelli's contraction Theorem and many applications see \cite{MilE-Y} and references therein.

Let $X=(\Ren,\|\cdot\|_K, \mu)$ be a normed probability space and $Y=(\Ren,\|\cdot\|_L)$ a normed space. Let $\nu$ be the push forward of $\mu$ by the natural map $\pi:supp(\mu)\to \Ren$ given by $\pi(x)=\frac{x}{\norm{x}_L}\norm{x}_K$. In the first part of this note we study how concentration properties of $\mu$ are inherited by $\nu$. We give a bound for the concentration function of $\nu$ in terms of the quotient
$\frac{\lambda m_K}{m_L} $, where $m_L$ and $m_K$ denote the medians of $\norm{\cdot}_L$ and $\normk{\cdot}$ resp. on $X$ with respect to $\mu$, and $\lambda$ is such that $\|\cdot\|_L\leq \|\cdot\|_K\leq \lambda \|\cdot\|_L$. In fact, if this quotient is bounded and $\mu$ verifies a concentration property, then also does $\nu$.

Examples as $\ell_2^n$ and $\ell_1^n$ show that the result allows one to transmit the concentration of measure phenomenon between two spaces quite far apart in the Banach-Mazur distance. As an application, we show that $\ell_\infty^n$ (and its high dimensional subspaces) are far (in the Banach--Mazur sense) from any space with a good concentration.

In the second part of the paper, we are given one norm $\|\cdot\|_L$ on $\Ren$ and two measures $\mu$ and $\nu$, with densities $d\mu=f(\norm{x}_L)dx$, $d\nu=g(\norm{x}_L)dx$ with respect to the Lebesgue measure. We compare the concentrations of $\mu$ and $\nu$. As an application, we prove concentration inequalities on $\ell_p^n$ for $1\leq p\leq 2$ (as those obtained in \cite{Ar-Vil}).
\section{Notations and previous results}

Throughout this note, all measures are regular Borel measures.

Let $T:X\to Y$ be a map from a measure space $(X,\mu)$ to a set $Y$. The push forward measure of $\mu$ by $T$ is defined by
$$\nu(A)=\mu(T^{-1}(A)),$$
for any Borel set $A\subset Y.$

A \emph{metric probability space} is a triple $(X,d,\mu)$, where $d$ and $\mu$ are a metric and a probability measure on $X$ respectively. For a set $A\subset X$ and $\epsilon>0$, the \emph{$\epsilon$-expansion of $A$} is defined by
$$
A_\epsilon^d=\{x\in X: d(x,a)\leq\epsilon \mbox{ for some } a\in A\}.
$$

The \emph{concentration function} of $(X,d,\mu)$ is defined by
\[\alpha_{\left(X,d,\mu\right)}(\epsilon)=
\sup\left\{
1-\mu(A_\epsilon^d): A \subset X\mbox{ measurable, }\mu(A)\ge1/2\right\}.\]


Recall that $m_f\in\Re$ is called a \emph{median} of a measurable function $f:X\to\Re$ on a probability space $(X,\mu)$ if $\mu(\{x\in X: f(x)\leq m_f\})\geq1/2$ and $\mu(\{x\in X: f(x)\geq m_f\})\geq1/2$. It is well known that if $f$ is Lipschitz with Lipschitz constant $L$, then for any $\epsilon >0$
\begin{equation}\label{eq:conmedian}
\mu(\{x\in X: |f(x)-m_f|\geq \epsilon \})\leq2\alpha_{\left(X,d,\mu\right)}(\epsilon/L).
\end{equation}

The following result shows a concentration inequality for the push forward measure induced by a Lipschitz map (see \cite{Led} for a proof).
\begin{prop}\label{prop:dec}
Let $\varphi$ be a L-Lipschitz map between two metric spaces $(X,d_1)$ and $(Y,d_2)$. Let $\mu$ be a Borel probability measure on $X$ and denote by $\nu$ the push forward measure of $\mu$ by $\varphi$ on $Y$. Then, for every $r>0$,
\begin{equation}\label{eqn:ineqpushfor}
    \alpha_{(Y,d_2,\nu)}(r)\leq\alpha_{(X,d_1,\mu)}(r/L).
\end{equation}
\end{prop}

This result can be used, for example, to push forward Gaussian measure to the Euclidean unit ball (see \cite{Led}).

Let $X,Y$ be $n$-dimensional normed spaces. The \emph{Banach-Mazur distance} between  $X$ and $Y$ is defined by

\begin{equation*}
    d_{BM}(X,Y)=\inf\{\lambda>0\ :\ L\subset TK\subset \lambda L \quad \mbox{for}\  T\in GL_n\},
\end{equation*}

where $K,L$ denote the unitary balls for the norms $\|\cdot\|_K$ and $\|\cdot\|_L$ respectively.

When the space $X$ is $\Ren$ endowed with a norm $\|\cdot\|_K$, and $d$ is the metric induced by this norm, we write for simplicity $A_\epsilon^K$ for the $\epsilon$-expansion of $A\subset\Ren$ and $\alpha_{\left(K,\mu\right)}$ for the concentration function $(\Ren,d,\mu)$.

\section{Concentration properties transferred through $\pi$}

Let $X=(\Ren,\|\cdot\|_K)$ and $Y=(\Ren,\|\cdot\|_L)$ be $n$-dimensional metric spaces (with the induced metrics), and let $\mu$ be a probability measure on $X$. Let $m_L$ and $m_K$ be medians of the functions $\|.\|_L:\Ren\to\Re$ and $\|.\|_K:\Ren\to\Re$, both with respect to $\mu$. In what follows, it is supposed that $\mu$ has a positive median $m_K$.

We consider the push forward $\nu$ of $\mu$ by the natural map $\pi:supp(\mu)\rightarrow \Ren$ given by $\pi(x)= \frac{x}{\norm{x}_L}\norm{x}_K$.
Assuming that $\|.\|_K\leq \|.\|_L\leq \lambda\|.\|_K$, it is easy to show that $\pi$ is $(2\lambda+1)$-Lipschitz which, by Proposition~\ref{prop:dec}, yields an upper bound on the concentration function of $\nu$ in terms of $\lambda$. Note that after a proper linear transformation applied to $L$, $\lambda$ is equal to the Banach--Mazur distance of $K$ and $L$. Our first result is a strengthening of this bound: we may use $\lambda\frac{m_K}{m_L}$ instead of $\lambda$.
\begin{thm}\label{thm:main}
Let $X=(\Ren,\|.\|_K)$ and $Y=(\Ren,\|.\|_L)$ be normed spaces with $\|.\|_K\leq\|.\|_L\leq \lambda \|.\|_K$. Let $\mu$ be a regular Borel probability measure on $X$ with concentration function $\alpha_{(K,\mu)}$. Let $\pi:supp(\mu)\rightarrow \Ren$ be the map $\pi(x)= \frac{x}{\norm{x}_L}\norm{x}_K$, and denote by $\alpha_{(L,\nu)}$ the concentration function of $\nu$, the push forward measure of $\mu$ by $\pi$. Then for every $\epsilon>0$ such that $16\alpha_{(K,\mu)}\left(\epsilon m_L/7\lambda m_K \right)\leq 1$ we have
\begin{equation*}
\alpha_{(L,\nu)}(\epsilon)\leq16\alpha_{(K,\mu)}\left(\frac{\varepsilon m_L}{14\lambda m_K }\right).
\end{equation*}
\end{thm}

For the proof, we quote the following lemma from \cite{Led}.

\begin{lem}
Let $\mu$ be a probability measure on the Borel sets of a metric space $(X,d)$ with concentration function $\alpha_{(X,d,\mu)}$. For any two non-empty Borel sets $A$ and $B$ in $X$,

\begin{equation}\label{eq:Ledoux}
    \mu(A)\mu(B)\leq 4\alpha_{(X,d,\mu)}(\dist(A,B)/2)
\end{equation}
where $\dist(A,B)=\inf\{d(x,y)\ :\ x\in A, y\in B\}.$
\end{lem}
\begin{proof}[Proof of Theorem~\ref{thm:main}]
Let $A\subset Y$ be a Borel set with $1/2\leq\nu(A)=\mu(\pi^{-1}(A))$, and let $\varepsilon>0$ be given.
Since
\begin{equation*}
 \|.\|_K\leq\|.\|_L\leq \lambda \|.\|_K,
\end{equation*}

it follows that $m_K\leq m_L\leq \lambda m_K$.

Let $0<\delta=\frac{\epsilon}{7m_K}$ and
denote by
$$G_L=\{x\in \Ren: (1-\delta)m_L< \|x\|_L<(1+\delta)m_L\}$$
$$G_K=\{x\in \Ren: (1-\delta)m_K< \|x\|_K<(1+\delta)m_K\}.$$

Let $J=\pi^{-1}(A)\cap G_L\cap G_K$. By (\ref{eq:conmedian}) we have that $$\mu(J)\geq1/2-2\alpha_{(K,\mu)}(\delta m_L/\lambda)-2\alpha_{(K,\mu)}(\delta m_K)\geq1/2-4\alpha_{(K,\mu)}(\delta m_L/\lambda).$$

We claim that

\begin{equation}\label{eq:Jcontpia}
J_{\frac{\delta m_L}{\lambda }}^K\subset\pi^{-1}\left(A_\epsilon^L\right)
\end{equation}
Indeed, let $x\in J_{\frac{\delta m_L}{\lambda }}^K$. Then, there exists $y\in J$ with $\norm{x-y}_K\leq{\frac{\delta m_L}{\lambda }}$. Since $y\in J$ we have that $\pi y\in A$, $|\norm{y}_K-m_K|<\delta m_K$ and $|\norm{y}_L-m_L|<\delta m_L$.  We will show that $\pi x\in A_{\varepsilon}^L$.
\begin{align*}
    \nonumber\norm{\pi x&-\pi y}_L =\normbl{\frac{x}{\norm{x}_L}\norm{x}_K-\frac{y}{\norm{y}_L}\norm{y}_K} \\
         &\leq\norm{x}_L\left|\frac{\norm{x}_K}{\norm{x}_L}-\frac{m_K}{m_L}\right|+\frac{m_K}{m_L}\norm{x-y}_L+\norm{y}_L\left|\frac{\norm{y}_K}{\norm{y}_L}-\frac{m_K}{m_L}\right|\\
     \nonumber&\leq\frac{1}{m_L}\Big[\left|\norm{x}_Km_L-\norm{x}_Lm_K\right|+m_K\lambda\normk{x-y}+\left|\norm{y}_Km_L-\norm{y}_Lm_K\right|\Big].
\end{align*}

Now,\\
    \begin{eqnarray*}
      \nonumber\left|\norm{y}_Km_L-\norm{y}_Lm_K\right| &\leq& m_L\left|\norm{y}_K-m_K \right|+m_K\left|\norm{y}_L-m_L \right| \\
       &<& m_L \delta m_K+m_K\delta m_L=2\delta m_Lm_K.
    \end{eqnarray*}

Finally, we note that\\
      \begin{align*}
         |\normk{x}m_L&-\norm{x}_Lm_K| \\
         &\leq m_L\norm{x-y}_K+\left|\norm{y}_Km_L-\norm{y}_Lm_K \right|+m_K\norm{x-y}_L \\
         &\leq m_L\frac{\delta m_L}{\lambda }+2\delta m_Lm_K+\lambda m_K\frac{\delta m_L}{\lambda }\\
         &\leq\delta m_Lm_K+2\delta m_Lm_K+\delta m_Lm_K= 4\delta m_Lm_K.
      \end{align*}
Thus, \\
    $$\norm{\pi x-\pi y}_L\leq4\delta m_K+\delta m_K+2\delta m_K=7\delta m_K,$$
and (\ref{eq:Jcontpia}) follows.

By (\ref{eq:Ledoux}) and (\ref{eq:Jcontpia}) we have
    $$\nu((A_\epsilon^L)^c)=\mu(\pi^{-1}(A_\epsilon^L)^c)\leq{\mu((J_{\frac{\delta m_L}{\lambda }}^K)^c)}\leq\frac{4\alpha_{(K,\mu)}(\frac{\delta m_L}{2\lambda })}{\mu(J)}\leq\frac{4\alpha_{(K,\mu)}(\frac{\delta m_L}{2\lambda})}{1/2-4\alpha_{(K,\mu)}(\frac{\delta m_L}{\lambda})},$$ and hence, for every $\epsilon>0$ such that $16\alpha_{(K,\mu)}\left(\epsilon m_L/7\lambda m_K \right)\leq 1$, $$\nu(A_\epsilon^c)\leq16\alpha_{(K,\mu)}\left(\frac{\varepsilon m_L}{14\lambda m_K}\right).$$

    Since the latter holds for every $A\subset Y$ with $\nu(A)\geq1/2$ we conclude
    $$\alpha_{(L,\nu)}(\epsilon)\leq16\alpha_{(K,\mu)}\left(\frac{\varepsilon m_L}{14\lambda m_K}\right),$$
    as we wanted.
\end{proof}

\begin{rem}
For a measure $\mu$ supported on $\partial K$ we can substitute $m_K=1$ and obtain  $\alpha_{(L,\nu)}(\epsilon)\leq16\alpha_{(K,\mu)}(\frac{\varepsilon m_L}{14\lambda})$ for $\nu$ the push-forward measure of $\mu$ into $\partial L$ by $\pi$.
\end{rem}

\begin{rem}
If $\mu=\lambda_K$ is the normalized Lebesgue measure restricted to $K$, then a median for $\norm{\cdot}_K$ is $2^{-1/n}$. In this case the bound $m_K\leq 1$ can be used to obtain $\alpha_{(L,\nu)}(\epsilon)\leq16\alpha_{(K,\mu)}(\frac{\varepsilon m_L}{14\lambda})$ without (asymptotically) losing much.
\end{rem}

\begin{rem}\label{rem:ineqforbeta}
If we define
\begin{equation*}
\beta=\beta((K,\mu),L):=\inf\{\lambda\frac{m_K}{m_{T^{-1}L}}\st L\subseteq TK\subseteq \lambda L\, \mbox{ for } T\in GL_n \}
\end{equation*}
where $m_{T^{-1}L}$ denotes a median of $\|\cdot\|_{T^{-1}L}$ in $K$ with respect to $\mu$ then
\begin{equation*}
\alpha_{(L,\tilde{\nu})}(\epsilon)\leq16\alpha_{(K,\mu)}\left(\frac{\epsilon}{14\beta}\right),
\end{equation*}
where $\tilde{\nu}$ is the push-forward of the measure $\mu$ with respect to $\tilde{\pi} x=\frac{Tx}{||Tx||_L}||x||_K$ and $T\in GL_n$ is the map where the infimum above is attained.
Note that $\beta\leq d_{BM}(X,Y)$.
\end{rem}

For simplicity we will denote $\beta((K,\lambda_K),L)$ by $\beta(K,L).$

This number, for the case $X=\ell_2^n$ endowed with the Haar measure on the Euclidean unit sphere $\mathbb{S}^{n-1}$, plays a central role in Milman's proof of Dvoretzky's theorem giving the dimension $k(Y)$ of $\epsilon$-almost Euclidean subspaces of $Y$: $k(Y)\geq c(\epsilon)\frac{n}{\beta^2}$ (see Theorem 4.2 in \cite{M-S}).
The inequality above gives us information about $\alpha_{(L,\nu)}$ as long as $\beta$ is bounded from above. This parameter $\beta$ shows in some sense how different $X$ and $Y$ are in terms of a given $\mu$. This result, after pushing forward the Haar measure on $\mathbb{S}^{n-1}$ (or the normalized Lebesgue measure on the Euclidean unit ball) to $\ell_1^n$, gives us a strong concentration of $\nu$ in $\ell_1^n$. The latter is possible because for this particular pair of spaces the parameter $\beta$ is bounded by a constant (although $d_{BM}(\ell_2^n,\ell_1^n)$ is exactly $\sqrt{n}$). See Remark \ref{rem:lpl2} for similar applications.

Note that, in case $\beta$ is bounded as before,  a \lq\lq good" concentration (normal, exponential, etc.) for $\alpha_{(K,\mu)}$ implies an asymptotically similar concentration for $\alpha_{(L,\nu)}$.

It is known that if the measure $\mu$ is $\log$-concave, combining Markov's inequality, Borell's lemma, and a Paley-Zygmund type argument, it can be seen that $\mathbb{E}\normk{\cdot}$ and $m_K$ are equivalent up to numeric constants. Thus, theorem \ref{thm:main} can be rewritten in the form obtaining an inequality of the type
\begin{equation*}
\alpha_{(L,\tilde{\nu})}(\epsilon)\leq16\alpha_{(K,\mu)}\left(\frac{\epsilon}{c\tilde{\beta}}\right),
\end{equation*}

where $c$ is a universal constant and
\begin{equation*}
\tilde{\beta}=\tilde{\beta}((K,\mu),L):=\inf\{\lambda\frac{\mathbb{E}\normk{\cdot}}{\mathbb{E}\|{\cdot}\|_{T^{-1}L}}\st L\subseteq TK\subseteq \lambda L\, \mbox{ for } T\in GL_n \}.
\end{equation*}

\begin{rem}\label{rem:lpl2}
A straight forward computation (see \cite{M-S}) shows that for $B_2^n$ endowed with the Lebesgue measure and $1\leq p<2$, $\tilde{\beta}(B_2^n,B_p^n)\leq b_p$ where $b_p$ depends only on $p$. Therefore, we obtain a concentration phenomenon on $B_p^n$ with respect to the push-forward of the Lebesgue measure on $B_2^n$ to $B_p^n$. This result is analogous to the one obtained in \cite{Ar-Vil} (with respect to a different measure) or in \cite{S-Z} (with respect to a different measure and a different metric).
\end{rem}

\begin{rem}
For $2<p<\infty$, the bound $\tilde{\beta}(B_2^n,B_p^n)\leq C_p n^{\frac{1}{2}-\frac{1}{p}}$ (see \cite{M-S}) implies the concentration inequality $\alpha_{(B_p^n,\nu)}(\epsilon)\leq C_3\exp\{-c_3\epsilon^2 n^{2/p}\}$. Compare with the result obtained in \cite{A-B-V} or \cite{G-M2} for any uniformly convex space (with respect to the Lebesgue measure).
\end{rem}

\begin{rem}
If we assume that $B_2$ is the John ellipsoid of a $0$-symmetric convex body $L$, then using $$\int_{S^{n-1}}\norml{x}d\sigma(x)\geq\int_{S^{n-1}}\|x\|_\infty d\sigma(x)$$ (cf. \cite{M-S} pp. 23-24) we obtain the general bound

$$\tilde{\beta}(B_2^n,L)\leq C\sqrt{\frac{n}{\log n}}.$$
\end{rem}
This estimate seems to be sharp.
\subsection{Applications}

Using the well known estimate for the mean of the $\sup$ norm with respect to the Haar measure on $\Se^{n-1}$
\[\int_{S^{n-1}}\|x\|_\infty d\sigma(x)\geq C\sqrt{\log n},\]
we get
\begin{equation}\label{eq:betal2linf}
\tilde{\beta}(B_2^n,B_\infty^n)\leq C \sqrt{\frac{n}{\log n}}.
\end{equation}

This bound does not yield any good concentration for the push forward of the Haar measure on $\Se^{n-1}$ to the cube $B_\infty^n$ by $\pi$.
In this section we show that no measure on $B_\infty^n$ has good concentration (Corollary~\ref{cor:farlinf}). we prove a more general result, which may be of independent interest: If an $n$-dimensional space with good concentration is close (in the Banach--Mazur sense) to a subspace of $\ell_\infty^N$ then $N$ is large (compared to $n$).
Compare this result with the one obtained in \cite{Tal1}, where a concentration property is obtained for $B_\infty^n$ endowed with a metric different from the supremum norm.

Recall that a \emph{$d$-embedding} between normed spaces is any lineal operator $T:X\rightarrow Y$ such that $a\|x\|_X\leq \|Tx\|_Y\leq b\|x\|_X$ with $a^{-1}b\leq d$.

\begin{thm}\label{thm:farlinf}
Let $X$ be an $n$-dimensional normed space, $K$ its closed unit ball and $\mu$ a symmetric Borel probability measure supported on $K$.

If there exists a $d$-embedding from $X$ into $\ell_\infty^N$, then, for any $0<\epsilon<1/d$ such that $\alpha_\mu(\epsilon)>0$, we have
$$N\geq\frac{1}{2}\alpha_\mu(\epsilon)^{-1}(1-\mu(d\epsilon K))$$
\end{thm}
\begin{proof}
Let $T:X\rightarrow \ell_\infty^N$ be an embedding such that $d^{-1}\|x\|\leq\|Tx\|_\infty\leq\|x\|$ for every $x\in X$. Consider the projection onto the $i^{th}$-coordinate $\pi_i:\ell_\infty^N\rightarrow \Re$ and the linear functional $f_i=\pi_i\circ T\, (1\leq i\leq N)$. For every $x\in X$, we have $|f_i(x)|\leq\|Tx\|_\infty\leq\|x\|$. Let $\epsilon>0$ and define, for $1\leq i\leq N$, the sets $A_i=\{x\in K\st |f_i(x)|\leq \epsilon\}.$ Using the symmetry of $\mu$ and the inclusion
$$\{x\in K\st f_i(x)\leq \epsilon\}\supset\{x\in K\st f_i(x)\leq0\}+\epsilon K,$$
we have $\mu\{x\in K\st f_i(x)>\epsilon\}\leq\alpha_\mu(\epsilon)$. Then
$$\mu\left(A_i^c\right)=2\mu\{x\in K\st f_i(x)>\epsilon\}\leq2\alpha_\mu(\epsilon)$$

for $1\leq i\leq N$. Therefore

$$\displaystyle{\mu\left(\bigcap_{i=1}^NA_i\right)\geq 1- \sum_{i=1}^N\mu\left(A_i^c\right)\geq1-2N\alpha_\mu(\epsilon)}.$$

Thus, for any $r\in(0,1)$ such that $\mu(rK)<1-2N\alpha_\mu(\epsilon)$, we have that $\displaystyle{\bigcap_{i=1}^NA_i\nsubseteq rK}$, and there exists $\displaystyle{x\in\bigcap_{i=1}^NA_i}$ with $\|x\|>r$. From this
$$r<\|x\|\leq d\|Tx\|_\infty\leq d\epsilon.$$

Consequently, $\mu(d\epsilon K)\geq 1-2N\alpha_\mu(\epsilon)$, as we wanted.
\end{proof}

\begin{cor}\label{cor:farlinf}
Let $\alpha_{(B_\infty^n,\nu)}(\epsilon)$ be the concentration function of a symmetric Borel probability measure $\nu$ on $B_\infty^n$. For any $0<\epsilon<1$,

    \begin{equation*}
    \alpha_{(B_\infty^n,\nu)}(\epsilon)\geq\frac{1}{2n}\left(1-\nu(\epsilon B_\infty^n)\right).
    \end{equation*}

\end{cor}
\begin{proof}
Consider the identity on $\ell_\infty^n$.
\end{proof}

\begin{thm}\label{thm:binffar}
Let $K\subset\Re^n$ be a convex body,  and $\mu$ a symmetric probability measure on $K$ with concentration function $\alpha_{(K,\mu)}(\epsilon)\leq Ce^{-c\epsilon^2n}$.  Assume that $\displaystyle{m_K>\frac{1}{2}\sqrt{\frac{\log(16C)}{\log{(64Cn)}}}}$. Then

$$\beta=\beta((K,\mu),B_\infty^n)\geq\frac{\sqrt{\log(16C)}}{28}\frac{\sqrt{cn}}{\log{(64Cn)}}$$
\end{thm}
\begin{proof}
Let $L=B_\infty^n$. Without loss of generality, we may assume that $B_\infty^n\subset K \subset \lambda B_\infty^n$ (as required in Theorem \ref{thm:main}) and $\beta=\lambda\frac{m_K}{m_L}$.

We also assume that

\begin{equation}\label{eq:hyp}
\displaystyle{\frac{\lambda}{m_L}<\frac{1}{14}\frac{\sqrt{cn}}{\sqrt{\log(64Cn)}}},
\end{equation}

otherwise the desired inequality follows.

Let $\epsilon_0=\frac{1}{2}\sqrt{\frac{\log(16C)}{\log{(64Cn)}}}$. Then $\mu(\normk{x}\leq \epsilon_0)\leq \mu(\normk{x}< m_K)\leq1/2.$

Clearly $\epsilon_0$ satisfies the hypothesis of Theorem \ref{thm:main}. Let $\alpha_{\nu}$ be the concentration function of $\nu$, the push forward measure of $\mu$ by $\pi:K\rightarrow B_\infty^n$. Therefore, combining this result and Corollary \ref{cor:farlinf} we have

\begin{eqnarray*}
  \frac{1}{4n} &\leq& \frac{1-\mu(\epsilon_0 K)}{2n}=\frac{1-\nu(\epsilon_0 B_\infty^n)}{2n}\leq \alpha_\nu(\epsilon_0)\\
    &\leq& 16\alpha_\mu\left(\epsilon_0\frac{m_L}{14m_K \lambda}\right)\leq16C\exp\left(-c\epsilon_0^2\frac{m_L^2}{14^2m_K^2\lambda^2}n\right)
\end{eqnarray*}

By (\ref{eq:hyp}),

$$\frac{1}{4n}\leq16 C\exp\left(-cn\frac{1}{2^2\beta^2}\frac{\log(16C)}{14^2\log(64Cn)}\right),$$

which implies

$$\beta\geq\frac{\sqrt{\log(16C)}}{28}\frac{\sqrt{cn}}{\log(64Cn)}$$

as we wanted.
\end{proof}

\begin{rem}
A similar bound can be obtained for $\tilde{\beta}$.
\end{rem}

\begin{rem}
If we consider $\mu$ the Lebesgue measure restricted to $K$, then $m_K=2^{-1/n}$ and the hypothesis on $m_K$ imposed in the previous theorem is trivially satisfied for all $n\geq1$.
\end{rem}

\begin{rem}
     Theorem \ref{thm:binffar} implies $d_{BM}(K,B_\infty^n)\geq\frac{\sqrt{cn}}{28}\frac{\sqrt{\log(16C)}}{\log(64Cn)}$. However, a simpler argument using Corollary \ref{cor:farlinf} and Proposition \ref{prop:dec} gives the slightly better bound  $d_{BM}(K,B_\infty^n)\geq\sqrt{\frac{cn}{\log (2Cn)}}$.
\end{rem}

\begin{rem}
It is unknown whether the bound obtained in the previous theorem is sharp. However the extremality of $B_2^n$ with respect to $B_\infty^n$ along with the inequality (\ref{eq:betal2linf}) seem to leave some room for improvement.
\end{rem}

\section{Concentration properties transferred between two given measures}
Theorem \ref{thm:main} allows us to transmit a concentration inequality from a given metric probability space to the push-forward measure given by the natural map $\pi$. However, this new measure might not be related to a measure given in the target space $Y$. This section is devoted to give a concentration inequality with respect to a given measure defined in the target space. Thus, given two metric probability spaces, we will investigate the relation between their respective concentration functions applying similar ideas to those used in the previous section.

For a given norm $\|\cdot\|_L$ on $\Ren$, let us consider $\mu$ and $\nu$ two probability measures with densities $d\mu=f(\norm{x}_L)dx$ and $d\nu=g(\norm{x}_L)dx$ with respect to the Lebesgue measure. It is not hard to see that for some function $u:[0,+\infty)\rightarrow [0,+\infty)$, $u(0)=0$, the map $U:supp(\mu)\rightarrow\Ren$ defined by $Ux=\frac{u(\norm{x})}{\norm{x}}x$ pushes forward $\mu$ into $\nu$. The map $U$ has finite Lipschitz constant if $u$ has \cite{Bo-Le}. In what follows, we will not use this fact, but only the the assumption of $\norm{u}_{Lip}$ being finite.

\begin{thm}\label{thm:main1}
Let $X=(\Ren,\norm{\cdot}_K),\ Y=(\Ren,\norm{\cdot}_L)$ be normed spaces such that $\norm{\cdot}_K\leq\norm{\cdot}_L\leq\lambda \norm{\cdot}_K$. Let $\mu$ and $\nu$ be probability measures on $X$ and $Y$ respectively, with densities $d\mu=f(\norm{x}_L)dx$ and $d\nu=g(\norm{x}_L)dx$ with respect to the Lebesgue measure. Denote their concentration functions by $\|\cdot\|_K$ and $\|\cdot\|_L$ respectively by $\alpha_{\left(K,\mu\right)}$ and $\alpha_{\left(L,\nu\right)}$. Then for every $\epsilon>0$ such that $8(\alpha_{\left(K,\mu\right)}(\varepsilon /7||u||_{Lip}\lambda)+\alpha_{\left(K,\mu\right)}(\varepsilon m/7||u||_{Lip}^2m_L))\leq 1$ we have
\begin{equation*}
\alpha_{\left(L,\nu\right)}(\epsilon)\leq16\alpha_{\left(K,\mu\right)}\left(\frac{\varepsilon }{14||u||_{Lip}\lambda}\right),
\end{equation*}
where $m_{L}$ denotes a median of the function $\norm{\cdot}_L:\Ren\rightarrow \Re$, $m$ denotes a median of the function $u(\norm{\cdot}):\Ren\rightarrow \Re$ both with respect to the measure $\mu$, and $||u||_{Lip}$ is the Lipschitz constant of $u$.
\end{thm}
\begin{proof}
Let $A\subset L$ be a Borel set with $1/2\leq\nu(A)=\mu(U^{-1}(A))$, and let $\varepsilon>0$ be given.
Let $0<\delta=\frac{\epsilon}{7m_L\norm{u}_{Lip}}$. Since $u(\norm{x}_L)\leq \norm{u}_{Lip}\norm{x}_L$, we have that $m\leq||u||_{Lip}m_L$. Let
$$G_1=\{x\in \Ren: (1-\delta)m_L< \|x\|_L<(1+\delta)m_L\}$$
$$G_2=\{x\in \Ren: (1-\delta)m< u(\|x\|_L)<(1+\delta)m\},$$

and $J=U^{-1}(A)\cap G_1\cap G_2$. By (\ref{eq:conmedian}) we have that $$\mu(J)\geq1/2-2\alpha_{\left(K,\mu\right)}(\delta m_L/\lambda)-2\alpha_{\left(K,\mu\right)}(\delta m/||u||_{Lip}).$$
We claim that

\begin{equation}\label{eq:Jcontpia1}
\displaystyle{J^K_{\frac{\delta m_L}{\lambda}}\subset U^{-1}\left(A^L_\epsilon\right)}.
\end{equation}
Indeed, let $x\in J^K_{\frac{\delta m_L}{\lambda}}$. Then, there is a $y\in J$ with $\|x-y\|_K\leq\frac{\delta m_L}{\lambda}$. Since $y\in J$ we have that $U(y)\in A$, $$|\norm{y}_L-m_L|<\delta m_L\quad \text{and}\quad |u(\norm{y}_L)-m|<\delta m.$$ We will show that $U x\in A^L_{\varepsilon}$.

\begin{align*}
   \norm{U(x)-U(y)}_L& =\norm{\frac{x}{\norm{x}_L}u(\norm{x}_L)-\frac{y}{\norm{y}_L}u(\norm{y}_L)}_L \\
    &\leq \frac{1}{m_L}\left[|u(\norm{x}_L)m_L-\norm{x}_Lm|+m\lambda\|x-y\|_K\right]\\
    &+\frac{1}{m_L}\left[|u(\norm{y}_L)m_L-\norm{y}_Lm|\right]
\end{align*}
On one hand,\\
    \begin{eqnarray*}
      \nonumber|u(\norm{y}_L)m_L-\norm{y}_Lm|&\leq& m_L\left|u(\norm{y}_L)-m \right|+m\left|\norm{y}_L- m_L\right| \\
       &\leq& 2m \delta m_L\leq2\delta m_L^2\norm{u}_{Lip}.
    \end{eqnarray*}
On the other hand,\\
      \begin{align*}
         |u(\norm{x}_L)&m_L-\norm{x}_Lm| \\
         &\leq m_L|u(\norm{x}_L)-u(\norm{y}_L)|+\left|u(\norm{y}_L)m_L-\norm{y}_Lm \right|+m\norm{x-y}_L \\
         &\leq m_L||u||_{Lip}|\|x\|_L-\|y\|_L|+2\delta m_L^2\norm{u}_{Lip}+m\lambda\|x-y\|_K\\
         &\leq m_L||u||_{Lip}\lambda\norm{x-y}_K+2\delta m_L^2\norm{u}_{Lip}+m\lambda\|x-y\|_K\\
         &\leq \delta m_L^2||u||_{Lip}+2\delta m_L^2\norm{u}_{Lip}+\delta m_L^2\norm{u}_{Lip}=4\delta m_L^2\norm{u}_{Lip}
      \end{align*}
Thus, \\
    $$\norm{U(x)-U(y)}_L\leq4\delta m_L\norm{u}_{Lip}+\delta m_L ||u||_{Lip}+2\delta m_L ||u||_{Lip}=7\delta m_L ||u||_{Lip},$$

and (\ref{eq:Jcontpia1}) follows. Now apply Lemma (\ref{eq:Ledoux}) to get
    $$\nu(A_\epsilon^c)=\mu(U^{-1}(A_\epsilon)^c)\leq{\mu((J_{\frac{\delta m_L}{ \lambda}})^c)}\leq\frac{4\alpha_{\left(K,\mu\right)}(\frac{\delta m_L}{2\lambda})}{\mu(J)}$$$$\leq\frac{4\alpha_{\left(K,\mu\right)}(\frac{\delta m_L}{2 \lambda})}{1/2-2\alpha_{\left(K,\mu\right)}(\delta m_L/\lambda)-2\alpha_{\left(K,\mu\right)}(\delta m/||u||_{Lip})}.$$ Hence, for $\epsilon>0$ such that
$8\left(\alpha_{\left(K,\mu\right)}\left(\frac{\varepsilon}{7||u||_{Lip}\lambda}\right)+\alpha_{\left(K,\mu\right)}\left(\frac{\varepsilon m}{7||u||_{Lip}^2m_L}\right)\right)\leq 1$
$$\nu(A_\epsilon^c)\leq16\alpha_{\left(K,\mu\right)}\left(\frac{\varepsilon }{14||u||_{Lip} \lambda}\right),$$
which concludes the proof.

\end{proof}

\begin{rem}
If we consider $X=(\Ren,\gamma_1^n)$ and $Y=(\Ren,\lambda_{B_1^n})$ both of them endowed with the distance induced by the $\ell_1$ norm, where $\gamma_1^n$ denotes the measure with density $\frac{1}{2^n}\exp{(-\norm{x}_1)}$ and $\lambda_{B_1^n}$ is the uniform measure restricted to $B_1^n$, then a simple computation (see \cite{Bo-Le} section 5) shows that $\norm{u}_{Lip}\approx\frac{1}{n}$ in the previous theorem. Thus, Theorem \ref{thm:main1} together with Talagrand's well known concentration inequality for $\gamma_1^n$ imply  $\alpha_{(B_1^n,\lambda_{B_1^n})}(\varepsilon)\leq K\exp\{-c\varepsilon^2 n\}$, where $c, K>0$ are universal constants (see \cite{Ar-Vil}).
\end{rem}

\begin{rem}
The same argument can be applied to $\gamma_p^n$ with density $c_p^{-1}e^{\norm{x}_p^p/p}dx$ and $\lambda_{B_p^n}$ where $c_p=2\Gamma(1+1/p)p^{1/p}$ ($1<p\leq 2$), to obtain $\alpha_{B_p^n,\lambda_{B_p^n}}(\varepsilon)\leq Kexp\{-c\epsilon^2 n\}.$
\end{rem}

\section*{Acknowledgement}
This work was carried out while the second named author was visiting the
Department of Mathematical Analysis at the University of Sevilla. The author would like to thank the Spanish Ministry of Science and Innovation grant MTM2009-08934 who partially funded his visit.

\bibliography{Art_Conc}
\bibliographystyle{plain}
\end{document}